\documentclass[12pt,reqno]{amsart}
\usepackage{amsmath, amsthm, amssymb}

\advance \topmargin by -\headheight
\advance \topmargin by -\headsep
     
\evensidemargin \oddsidemargin
\newtheorem{theorem}{Theorem}[section]
\newtheorem{lemma}[theorem]{Lemma}
\newtheorem{corollary}[theorem]{Corollary}

\theoremstyle{definition}

\theoremstyle{remark}

\numberwithin{equation}{section}

\def\bfa{{\mathbf a}}
\def\bfb{{\mathbf b}}

\def\bfu{{\mathbf u}}

\def\bfx{{\mathbf x}}

\def\calC{{\mathcal C}}

\def\calI{{\mathcal I}}

\def\calS{{\mathcal S}}
\def\calT{{\mathcal T}} \def\Ttil{{\widetilde T}}

\def\dbN{{\mathbb N}}
\def\dbR{{\mathbb R}}
\def\dbZ{{\mathbb Z}}\def\dbQ{{\mathbb Q}}

\def\grA{{\mathfrak A}}

\def\grI{{\mathfrak I}}

\def\grm{{\mathfrak m}}\def\grM{{\mathfrak M}}
\def\grn{{\mathfrak n}}

\def\grK{{\mathfrak K}}

\def\alp{{\alpha}} \def\bfalp{{\boldsymbol \alpha}}
\def\bet{{\beta}}  \def\bfbet{{\boldsymbol \beta}}
\def\gam{{\gamma}} 
\def\del{{\delta}} \def\Del{{\Delta}}

\def\tet{{\theta}}  \def\Tet{{\Theta}}
\def\kap{{\kappa}}
\def\lam{{\lambda}}

\def\d{{\partial}}
\def\eps{\varepsilon}

\def\le{\leqslant} \def\ge{\geqslant}

\def\d{{\,{\rm d}}}

\begin{document}
\title[Odd cubic Weyl sums]{Mean value estimates for odd cubic Weyl sums}
\author[Trevor D. Wooley]{Trevor D. Wooley}
\address{School of Mathematics, University of Bristol, University Walk, Clifton, Bristol BS8 1TW, United 
Kingdom}
\email{matdw@bristol.ac.uk}
\date{}
\begin{abstract} We establish an essentially optimal estimate for the ninth moment of the exponential sum 
having argument $\alp x^3+\bet x$. The first substantial advance in this topic for over $60$ years, this 
leads to improvements in Heath-Brown's variant of Weyl's inequality, and other applications of Diophantine 
type.\end{abstract}
\maketitle

\section{Introduction} This memoir concerns the mean values
$$I_s(X)=\int_0^1\int_0^1\Bigl| \sum_{1\le x\le X}e(\alp x^3+\bet x)\Bigr|^s\d\alp\d\bet ,$$
where $e(z)=e^{2\pi iz}$. Estimates for $I_s(X)$ make an appearance in the literature as early as 1947, 
when L.-K. Hua \cite[Lemma 4.3 and Theorem 6]{Hua1947} showed that
\begin{equation}\label{1.1}
I_6(X)\ll X^3(\log 2X)^9\quad \text{and}\quad I_{10}(X)\ll_\eps X^{6+\eps}.
\end{equation}
These mean values have more recently been applied to obtain improvements in Weyl's inequality and 
Waring's problem (see \cite{Bok1994, HB1988}), and also in investigations concerning the integral solubility 
of diagonal cubic equations subject to a linear slice (see \cite{BR2014}). Presumably, one should in general 
have the upper bound $I_s(X)\ll X^{s/2}+X^{s-4}$, but hitherto, the best available estimates for 
$I_s(X)$ are little better than those obtained from Hua's bounds (\ref{1.1}) via H\"older's inequality. By 
applying the cubic case of the main conjecture in Vinogradov's mean value theorem, recently established in 
\cite{Woo2014b}, we are now able to obtain estimates for $I_s(X)$ substantially sharper than these earlier 
bounds.

\begin{theorem}\label{theorem1.1}
For each $\eps>0$, one has $I_8(X)\ll X^{13/3+\eps}$ and $I_9(X)\ll X^{5+\eps}$.
\end{theorem}

By orthogonality, the mean value $I_6(X)$ counts the number of integral solutions of the system
$$\sum_{i=1}^3(x_i^3-y_i^3)=\sum_{i=1}^3(x_i-y_i)=0,$$
with $1\le x_i,y_i\le X$ $(1\le i\le 3)$. These simultaneous equations, defining the so-called Segre cubic 
(see \cite{Seg1887}) has been the focus of vigorous investigation in recent years. Vaughan and Wooley 
\cite{VW1995} showed that
\begin{equation}\label{1.2}
I_6(X)=6X^3+U(X),
\end{equation}
where $U(X)\asymp X^2(\log X)^5$, and de la Bret\`eche \cite{dlB2007} has obtained an asymptotic 
formula for $U(X)$ of the shape $U(X)\sim CX^2(\log X)^5$, for a suitable positive constant $C$. By 
interpolating between (\ref{1.2}) and the $10^{\text{th}}$-moment of Br\"udern and Robert 
\cite[Theorem 2]{BR2014}, one would obtain the estimates
$$I_8(X)\ll X^{9/2}(\log X)^{-1}\quad \text{and}\quad I_9(X)\ll X^{21/4}(\log X)^{-3/2}.$$ 
These estimates are sharper by a factor $X^\eps$ than the estimates that would stem from Hua's bounds 
(\ref{1.1}), whereas our new estimates save $X^{1/6-\eps}$ and $X^{1/4-\eps}$ in the respective cases. 
Indeed, our new bound $I_9(X)\ll X^{5+\eps}$ falls short of the best possible bound $I_9(X)\ll X^5$ only 
by a factor $X^\eps$.\par

The estimates recorded in Theorem \ref{theorem1.1} are consequences of a minor arc bound that will likely 
be of greater utility than the former in applications of the Hardy-Littlewood method. In order to describe 
bounds of this type, we must introduce some additional notation. When $Q$ is a real number with 
$1\le Q\le X^{3/2}$, we define the major arcs $\grM(Q)$ to be the union of the intervals
$$\grM(q,a)=\{ \alp \in [0,1): |q\alp-a|\le QX^{-3}\},$$
with $0\le a\le q\le Q$ and $(a,q)=1$. We then define the complementary set of minor arcs $\grm(Q)$ by 
putting $\grm(Q)=[0,1)\setminus \grM(Q)$. Finally, we define the exponential sum 
$g(\alp,\bet)=g(\alp,\bet;X)$ by
\begin{equation}\label{1.3}
g(\alp,\bet;X)=\sum_{1\le x\le X}e(\alp x^3+\bet x),
\end{equation}
and define $I_s^*(X;Q)$ for $s\in \dbN$ by putting
\begin{equation}\label{1.4}
I_s^*(X;Q)=\int_0^1 \int_{\grm(Q)}|g(\alp,\bet)|^s\d\alp \d\bet .
\end{equation}

\begin{theorem}\label{theorem1.2}
Let $Q$ be a real number with $1\le Q\le X$. Then for each $\eps>0$, one has the estimates
$$I_{10}^*(X;Q)\ll X^{6+\eps}Q^{-1/3}\quad \text{and}\quad I_{12}^*(X;Q)\ll X^{8+\eps}Q^{-1}.$$
\end{theorem}

When $X^{3/4}\le Q\le X^{4/5}$, one finds from Br\"udern and Robert \cite[Theorem 2]{BR2014} that 
$I_{10}^*(X;Q)\ll X^6(\log X)^{-2}$, which saves a factor $(\log X)^2$ over the lower bound of order 
$X^6$ for the corresponding major arc estimate. Theorem \ref{theorem1.2}, meanwhile, would save a 
power of $X$. Indeed, since $I_9(X)\gg X^5$, the bound $I_{12}^*(X;X)\ll X^{7+\eps}$, that stems from 
Theorem \ref{theorem1.2}, can be construed as supplying a Weyl estimate $g(\alp,\bet)\ll X^{2/3+\eps}$ 
on average for $\alp\in \grm(X)$. A direct application of Weyl's inequality (see \cite[Lemma 2.4]{Vau1997}) 
would show only that $g(\alp,\bet)\ll X^{3/4+\eps}$.\par

We would argue that the progress represented in our improved estimates for moments of $g(\alp,\bet)$ 
justifies an account based on its merit alone. However, we take this opportunity to record an application of 
Theorem \ref{theorem1.2} to Heath-Brown's variant of Weyl's inequality. In this context, when $k$ is a 
natural number, we consider the exponential sum $f(\alp)=f_k(\alp;X)$ defined by
$$f_k(\alp;X)=\sum_{1\le x\le X}e(\alp x^k).$$

\begin{theorem}\label{theorem1.3}
Let $k\ge 6$, and suppose that $\alp\in \dbR$, $a\in \dbZ$ and $q\in \dbN$ satisfy $(a,q)=1$ and 
$|\alp-a/q|\le q^{-2}$. Then for each $\eps>0$, one has
$$f_k(\alp;X)\ll X^{1+\eps}\Tet^{2^{-k}}+X^{1+\eps}(\Tet/X)^{\frac{2}{3}2^{-k}},$$
where $\Tet=q^{-1}+X^{-3}+qX^{-k}$.
\end{theorem}

The conclusion of \cite[Theorem 1]{HB1988} delivers a bound analogous to that of Theorem 
\ref{theorem1.3} of the shape
\begin{equation}\label{1.5}
f_k(\alp;X)\ll X^{1+\eps}(X\Tet)^{\frac{4}{3}2^{-k}}.
\end{equation}
We note that Boklan \cite{Bok1994} has applied Hooley $\Del$-functions to replace the factor $X^\eps$ 
here by a power of $\log X$. A comparison between these estimates is perhaps not so transparent. Suppose 
then that $\tet$ is a real number with $0\le \tet\le k/2$, and that $a\in \dbZ$ and $q\in \dbN$ satisfy 
$(a,q)=1$ and $q+X^k|q\alp-a|\asymp X^\tet$. It is a consequence of Dirichlet's theorem on Diophantine 
approximation that, given $\alp \in \dbR$, one can choose $a$ and $q$ in such a manner for some 
$\tet\le k/2$. One finds that the conclusion of Theorem \ref{theorem1.3} has strength equal to that of 
Heath-Brown's bound for $3\le \tet\le k/2$. When $2<\tet<3$, meanwhile, Theorem \ref{theorem1.3} 
delivers the bound $f_k(\alp;X)\ll X^{1+\eps-\frac{2}{3}(1+\tet)2^{-k}}$, which is superior both to the 
bound $f_k(\alp;X)\ll X^{1+\eps-\frac{4}{3}(\tet-1)2^{-k}}$ stemming from Heath-Brown's bound 
(\ref{1.5}), and also to the classical version of Weyl's inequality, which yields 
$f_k(\alp;X)\ll X^{1+\eps-2^{1-k}}$ (see \cite[Lemma 2.4]{Vau1997}). Both Theorem \ref{theorem1.3} 
and (\ref{1.5}) are weaker than the classical version of Weyl's inequality for $0<\tet<2$, though Theorem 
\ref{theorem1.3} remains non-trivial throughout this range.\par

By a standard transference principle (see Exercise 2 of \cite[\S2.8]{Vau1997}), the conclusion of Theorem 
\ref{theorem1.3} may be extended to a superficially more general conclusion which improves the first 
assertion of \cite[Theorem 1]{HB1988} for ranges of parameters analogous to those discussed above.

\begin{corollary}\label{corollary1.4}
Let $k\ge 6$, and suppose that $\alp\in \dbR$, $a\in \dbZ$ and $q\in \dbN$ satisfy $(a,q)=1$. Then one 
has
$$f_k(\alp;X)\ll X^{1+\eps}\Phi^{2^{-k}}+X^{1+\eps}(\Phi/X)^{\frac{2}{3}2^{-k}},$$
where
$$\Phi=(q+X^k|q\alp-a|)^{-1}+X^{-3}+(q+X^k|q\alp-a|)X^{-k}.$$
\end{corollary}

We finish by directing the reader to a couple of immediate applications of Theorems \ref{theorem1.1} and 
\ref{theorem1.2}, the proofs of which, amounting to routine applications of the circle method, we omit. 
First we consider the solubility of diagonal cubic equations constrained by a linear slice. When $s\in \dbN$, 
consider fixed integers $a_j,b_j$ $(1\le j\le s)$. Define $N(B)=N(B;\bfa,\bfb)$ to be the number of integral 
solutions of the simultaneous equations
\begin{equation}\label{1.6}
\sum_{j=1}^sa_jx_j^3=\sum_{j=1}^sb_jx_j=0,
\end{equation}
with $|x_j|\le B$ $(1\le j\le s)$. Then by incorporating the $10^{\rm{th}}$-moment estimate of 
Theorem \ref{theorem1.2} into the methods described in Br\"udern and Robert \cite[\S8]{BR2014}, one 
obtains the following conclusion.

\begin{theorem}\label{theorem1.5} Let $s\ge 10$ and suppose that $a_j\ne 0$ $(1\le j\le s)$. Suppose in 
addition that the pair of equations (\ref{1.6}) has non-singular solutions both in $\dbR$ and in $\dbQ_p$ 
for each prime number $p$. Then there are positive numbers $\calC(\bfa,\bfb)$ and $\del$ for which
$$N(B;\bfa,\bfb)=\calC(\bfa,\bfb)B^{s-4}+O(B^{s-4-\del}).$$
\end{theorem}

Br\"udern and Robert \cite[Theorem 1]{BR2014} establish precisely this conclusion as the cubic case of a 
more general result, though with the error term $O(B^{s-4-\del})$ replaced by 
$O(B^{s-4}(\log B)^{-2})$.  We offer no details of the proof of Theorem \ref{theorem1.5}, since the first 
estimate of Theorem 1.2 may be substituted for \cite[Theorem 2]{BR2014} in the argument of 
\cite[\S8]{BR2014}, without complication\footnote{The author is very grateful to J\"org Br\"udern and 
Olivier Robert for supplying an advance copy of their joint paper \cite{BR2014}, reference to which provides 
an excellent framework for the proof of this result.}.\par

Next, consider a fixed natural number $k$, and fixed coefficients $a_0,\ldots ,a_s\in \dbZ\setminus \{0\}$ 
and $b_1,\ldots ,b_s\in \dbZ$. By more fully exploiting the potential of the $9^{\rm{th}}$ moment 
estimate of Theorem \ref{theorem1.1}, it would be possible to apply the circle method to the problem of 
representing large positive integers $n$ in the shape
\begin{equation}\label{1.7}
F(x_1,\ldots ,x_s)+w^k=n,
\end{equation}
for the class of non-degenerate cubic forms $F$ of the shape
$$F(\bfx)=a_0(b_1x_1+\ldots +b_sx_s)^3+a_1x_1^3+\ldots +a_sx_s^3.$$
Thus, provided only that $s\ge 8$, for any $k\ge 1$, one can show that all sufficiently large natural 
numbers $n$ subject to the necessary congruence conditions are represented in the form (\ref{1.7}).\par

The strategy for proving this assertion is to replace (\ref{1.7}) by the equivalent system of equations
$$\left. \begin{aligned}
a_0x_0^3+a_1x_1^3+\ldots +a_sx_s^3&=n-w^k\\
x_0-b_1x_1-\ldots -b_sx_s&=0
\end{aligned} \right\}.$$
The analysis of this system is achieved by H\"olderising the associated exponential sums in order to utilise 
the mean value estimate
$$\int_0^1\int_0^1 |g(\alp,\bet)|^{s+1}\d\alp\d\bet \ll X^{s-3+\eps},$$
valid for $s\ge 8$, together with a pedestrian application of Weyl's inequality for the exponential sum over 
the $k$th power $w^k$. A routine treatment of the major arc contribution completes the analysis.\par

Throughout this paper, whenever $\eps$ appears in a statement, we assert that the statement holds for each 
$\eps>0$. Implicit constants in Vinogradov's notation $\ll$ and $\gg$ may depend on $\eps$, and other 
ambient exponents such as $k$, but not on the main parameter $X$. Finally, we write $\|\tet\|$ for 
$\underset{m\in \dbZ}{\min}|\tet -m|$.\par

\section{The basic mean value estimate} Our starting point for the proof of Theorems \ref{theorem1.1} 
and \ref{theorem1.2} is the mean value estimate supplied by the cubic case of the main conjecture in 
Vinogradov's mean value theorem, established in our very recent work \cite[Theorem 1.1]{Woo2014b}. 
When $k$ and $s$ are natural numbers, and $X$ is a large real number, denote by $J_{s,k}(X)$ the 
number of integral solutions of the system
$$x_1^j+\ldots +x_s^j=y_1^j+\ldots +y_s^j\quad (1\le j\le k),$$
with $1\le x_i,y_i\le X$ $(1\le i\le s)$. Then \cite[Theorem 1.1]{Woo2014b} shows that
\begin{equation}\label{2.1}
J_{s,3}(X)\ll X^\eps (X^s+X^{2s-6}).
\end{equation}
We transform this estimate into a bound for the $12$-th moment of $g(\alp,\bet)$ restricted to the set of 
minor arcs $\grm(Q)$ defined in the preamble to the statement of Theorem \ref{theorem1.2}. In this 
section we prove a number of mean value estimates for the exponential sum $g(\alp,\bet)$ defined in 
(\ref{1.3}), beginning with a mean value of the type (\ref{1.4}).

\begin{theorem}\label{theorem2.1}
Suppose that $Q$ is a positive number with $Q\asymp X$. Then for each $\eps>0$, one has 
$I_{12}^*(X;Q)\ll X^{7+\eps}$.
\end{theorem}

\begin{proof} When $k\in \dbN$, write
$$f(\bfalp)=\sum_{1\le x\le X}e(\alp_1x+\ldots +\alp_kx^k)$$
and
$$F(\bfbet ,\tet)=\sum_{1\le x\le X}e(\bet_1 x+\ldots +\bet_{k-2}x^{k-2}+\tet x^k).$$
Then it follows from orthogonality that
$$J_{s,k}(X)=\int_{[0,1)^k}|f(\bfalp)|^{2s}\d\bfalp .$$
In addition, write $\grn_k(Q)$ for the set of real numbers $\alp\in [0,1)$ having the property that, 
whenever $q\in \dbN$ and $\| q\alp\|\le QX^{-k}$, then $q>Q$. Then the argument of the proof of 
\cite[Theorem 2.1]{Woo2012b} leading to the penultimate display of that proof yields the estimate
\begin{equation}\label{2.2}
\int_{\grn_k(Q)}\int_{[0,1)^{k-2}}|F(\bfbet,\tet)|^{2s}\d\bfbet \d\tet \ll 
X^{k-2}(\log X)^{2s+1}J_{s,k}(2X).
\end{equation}
By specialising to the case $k=3$ and $s=6$, we therefore deduce from (\ref{2.1}) that
$$\int_0^1\int_{\grm (Q)}|g(\alp,\bet)|^{12}\d\alp \d\bet \ll X(\log X)^{13}J_{6,3}(2X)\ll X^{7+\eps}.
$$
This completes the proof of the lemma.
\end{proof}

We remark that a more careful analysis of the proof of \cite[Theorem 2.1]{Woo2012b} would reveal that, 
without restriction on $Q$, one may replace the estimate (\ref{2.2}) by the bound
$$\int_{\grn_k(Q)}\int_{[0,1)^{k-2}}|F(\bfbet,\tet)|^{2s}\d\bfbet \d\tet \ll 
X^{k-1+\eps}(Q^{-1}+X^{-1}+QX^{-k})J_{s,k}(2X).$$
Such an estimate would suffice to establish the bound $I_{12}^*(X;Q)\ll X^{8+\eps}Q^{-1}$. We will 
recover this estimate from Theorem \ref{theorem2.1} and Lemma \ref{lemma2.3} below in a manner that 
will likely prove more transparent for the reader.\par

By way of comparison, it follows from \cite[Theorem 6]{Hua1947} that $I_{10}(X)\ll X^{6+\eps}$. 
Applying this estimate in combination with Weyl's inequality (see \cite[Lemma 2.4]{Vau1997}) when 
$Q\asymp X$, one would obtain the upper bound
$$I_{12}^*(X;Q)\ll \Bigl( \sup_{\alp\in \grm(Q)}|g(\alp,\bet)|\Bigr)^2I_{10}(X)\ll X^{15/2+\eps},$$
in place of the conclusion of Theorem \ref{theorem2.1}. The superiority of our new estimate is clear.\par

We next establish some auxiliary major arc estimates. It is useful in this context to introduce some additional 
notation. We define the function $\Psi(\alp)$ for $\alp\in [0,1)$ by putting
$$\Psi(\alp)=(q+X^3|q\alp -a|)^{-1},$$
when $\alp\in \grM(q,a)\subseteq \grM(\tfrac{1}{2}X^{3/2})$, and otherwise by taking $\Psi(\alp)=0$.

\begin{lemma}\label{lemma2.2} Let $Q$ be a positive number with $Q\asymp X$, and suppose that 
$\alp \in \grM(Q)$. Then for each $\eps>0$, one has
$$\int_0^1|g(\alp,\bet)|^4\d\bet \ll X^{3+\eps}\Psi(\alp).$$
\end{lemma}

\begin{proof}
By orthogonality, one has
\begin{align*}
\int_0^1|g(\alp,\bet)|^4\d\bet &=\sum_{\substack{1\le x_1,x_2,x_3,x_4\le X\\ x_1+x_2=x_3+x_4}}
e((x_1^3+x_2^3-x_3^3-x_4^3)\alp )\\
&=\sum_{\substack{1\le x_1,x_2,x_3\le X\\ 1\le x_1+x_2-x_3\le X}}
e(-3(x_1+x_2)(x_1-x_3)(x_2-x_3)\alp).
\end{align*}
The change of variables
$$u_1=x_2-x_3,\quad u_2=x_1-x_3,\quad u_3=x_1+x_2$$
therefore reveals that
$$\int_0^1|g(\alp,\bet)|^4\d\bet =\sum_{-X<u_1,u_2,u_3\le 2X}e(-3u_1u_2u_3\alp),$$
in which the summation over $\bfu$ is subject to the condition that each of
$$u_2+u_3-u_1,\quad u_3+u_1-u_2,\quad u_3-u_1-u_2\quad \text{and}\quad u_1+u_2+u_3$$
is even, and lies in the interval $[1,2X]$. For a fixed choice of $u_1$ and $u_2$, the sum over $u_3$ 
consequently amounts either to an empty sum, or else to a sum over an arithmetic progression modulo $2$ 
lying in an interval of length at most $3X$. Thus we deduce that
\begin{equation}\label{2.3}
\int_0^1|g(\alp,\bet)|^4\d\bet \ll \sum_{1\le u_1,u_2\le 2X}\min \{ X,\|6\alp u_1u_2\|^{-1}\}.
\end{equation}

\par Suppose that $a\in \dbZ$ and $q\in \dbN$ satisfy the conditions $(a,q)=1$ and 
$|\alp-a/q|\le q^{-2}$. That such a rational approximation exists is a consequence of Dirichlet's theorem. 
Then by making use of a divisor function estimate together with a standard reciprocal sums lemma (see, for 
example \cite[Lemma 2.2]{Vau1997}), one deduces from (\ref{2.3}) that
\begin{align*}
\int_0^1|g(\alp,\bet)|^4&\ll X^\eps \sum_{1\le y\le 24X^2}\min\{X^3/y,\|\alp y\|^{-1}\}\\
&\ll X^{3+\eps}(q^{-1}+X^{-1}+qX^{-3}).
\end{align*}
Hence, by a standard transference principle (see Exercise 2 of \cite[\S2.8]{Vau1997}), one finds that 
whenever $\alp\in [0,1)$, $b\in \dbZ$ and $r\in \dbN$ satisfy $(b,r)=1$, then
\begin{equation}\label{2.4}
\int_0^1|g(\alp,\bet)|^4\d\bet \ll X^{3+\eps}(\lam^{-1}+X^{-1}+\lam X^{-3}),
\end{equation}
where $\lam=r+X^3|r\alp-b|$.\par

Suppose now that $\alp \in \grM(q,a)\subseteq \grM$. Then we have
$$q+X^3|q\alp -a|\ll X,$$
and thus it follows from (\ref{2.4}) that
$$\int_0^1|g(\alp,\bet)|^4\d\bet \ll X^{3+\eps}(X^{-1}+\Psi(\alp))\ll X^{3+\eps}\Psi(\alp).$$
This completes the proof of the lemma. 
\end{proof}

\begin{lemma}\label{lemma2.3} Suppose that $Q$ is a positive number with $Q\asymp X$. Then for each $\eps>0$, one has
$$\int_0^1\int_{\grM(Q)}|g(\alp,\bet)|^8\d\alp\d\bet \ll X^{4+\eps}.$$
\end{lemma}

\begin{proof} Suppose that $(\alp,\bet)\in [0,1)^2$, and that $a\in \dbZ$ and $q\in \dbN$ satisfy 
$(a,q)=1$ and $|\alp-a/q|\le q^{-2}$. Then it follows from Weyl's inequality (see 
\cite[Lemma 2.4]{Vau1997}) that
$$|g(\alp,\bet)|\ll X^{1+\eps}(q^{-1}+X^{-1}+qX^{-3})^{1/4}.$$
By applying the same transference principle that delivered (\ref{2.4}), we therefore deduce that when 
$\alp\in \grM(q,a)\subseteq \grM(Q)$, one has
\begin{equation}\label{2.5}
|g(\alp,\bet)|^4\ll X^{4+\eps}\Psi(\alp).
\end{equation}
By combining this estimate with the conclusion of Lemma \ref{lemma2.2}, therefore, we find that
$$\int_0^1|g(\alp,\bet)|^8\d\bet \ll X^{4+\eps}\Psi(\alp)\int_0^1|g(\alp,\bet)|^4\d\bet 
\ll X^{7+2\eps}\Psi(\alp)^2.$$
Consequently, we obtain the estimate
\begin{align*}
\int_0^1\int_{\grM(Q)}|g(\alp,\bet)|^8\d\alp\d\bet &\ll X^{7+\eps}
\sum_{1\le q\le Q}\sum_{a=1}^qq^{-2}\int_{-1/2}^{1/2}(1+X^3|\gam|)^{-2}\d\gam \\
&\ll X^{4+\eps}\sum_{1\le q\le Q}q^{-1}\ll X^{4+2\eps}.
\end{align*}
This completes the proof of the lemma.
\end{proof}

By utilising the conclusions of Lemma \ref{lemma2.3} and Theorem \ref{theorem2.1}, we obtain the mean 
value estimates recorded in Theorem \ref{theorem1.1}.

\begin{proof}[The proof of Theorem \ref{theorem1.1}] The estimate $I_6(X)\ll X^{3+\eps}$ is essentially 
classical (see \cite[Lemma 5.2]{Hua1947}). By combining this estimate with Theorem \ref{theorem2.1} via 
Schwarz's inequality, one finds that
\begin{equation}\label{2.6}
I_9^*(X;X)\le (I_{12}^*(X;X))^{1/2}(I_6(X))^{1/2}\ll (X^{7+\eps})^{1/2}
(X^{3+\eps})^{1/2}=X^{5+\eps}.
\end{equation}
Meanwhile, the trivial estimate $|g(\alp,\bet)|\le X$ combines with Lemma \ref{lemma2.3} to deliver the 
bound
\begin{equation}\label{2.7}
\int_0^1\int_{\grM(X)}|g(\alp,\bet)|^9\d\alp\d\bet \le X\int_0^1\int_{\grM(X)}|g(\alp,\bet)|^8\d\alp\d\bet 
\ll X^{5+\eps}.
\end{equation}
Since $[0,1)$ is the union of $\grM(X)$ and $\grm(X)$, the upper bound $I_9(X)\ll X^{5+\eps}$ follows 
by combining (\ref{2.6}) and (\ref{2.7}). Finally, by interpolating between the bound just obtained and 
Hua's estimate $I_6(X)\ll X^{3+\eps}$ via H\"older's inequality, one obtains
$$I_8(X)\le (I_6(X))^{1/3}(I_9(X))^{2/3}\ll (X^{3+\eps})^{1/3}(X^{5+\eps})^{2/3}=X^{13/3+\eps}.
$$
This completes the proof of Theorem \ref{theorem1.1}.
\end{proof}

Our last task in this section is that of establishing the minor arc bounds recorded in Theorem 
\ref{theorem1.2}.

\begin{proof}[The proof of Theorem \ref{theorem1.2}]
Suppose that $Q$ is a real number with $1\le Q\le X$, and write $\grK(Q)=\grM(X)\setminus \grM(Q)$. 
Then since $\grm(Q)$ is the union of $\grm(X)$ and $\grK(Q)$, one finds that
$$I_{12}^*(X;Q)\le I_{12}^*(X;X)+\Bigl( \sup_{\alp \in \grK(Q)}|g(\alp,\bet)|\Bigr)^4\int_0^1 
\int_{\grM(X)}|g(\alp,\bet)|^8\d\alp \d\bet .$$ 
When $\alp\in \grM(q,a)\cap \grK(Q)$, it follows that $q+X^3|q\alp-a|>Q$. Thus we deduce from 
(\ref{2.5}) that
$$\sup_{\alp \in \grK(Q)}|g(\alp,\bet)|\ll X^{1+\eps}(Q^{-1}+X^{-1})^{1/4}\ll 
X^{1+\eps}Q^{-1/4}.$$
Consequently, we find from Theorem \ref{theorem2.1} and Lemma \ref{lemma2.3} that
$$I_{12}^*(X;Q)\ll X^{7+\eps}+(X^{4+\eps}Q^{-1})(X^{4+\eps})\ll X^{8+2\eps}Q^{-1}.$$
This confirms the second estimate recorded in Theorem \ref{theorem1.2}. For the first, we apply H\"older's 
inequality to interpolate between the bound just obtained, and the second estimate asserted by Theorem 
\ref{theorem1.1}. Thus one has
\begin{align*}
I_{10}^*(X;Q)&\le (I_9(X))^{2/3}(I_{12}^*(X;Q))^{1/3}\\
&\ll (X^{5+\eps})^{2/3}(X^{8+\eps}Q^{-1})^{1/3}
=X^{6+\eps}Q^{-1/3}.
\end{align*}
This completes the proof of Theorem \ref{theorem1.2}.
\end{proof}

\section{A variant of Weyl's inequality} We turn in this section to the proof of Theorem \ref{theorem1.3}, 
and begin by recalling the key elements of the work of Heath-Brown \cite{HB1988} concerning a hybrid of 
the methods of Weyl and of Vinogradov. For the present, suppose that $k\ge 4$, and consider the 
exponential sum $f(\alp)=f_k(\alp;X)$. For each integer $m$, let $\grI(m)$ denote the real interval 
$[mX^{-3},(m+1)X^{-3})$. Given a real number $x$, we then denote by $m=m(x)$ the integer for which 
$x\in \grI(m)$, and we put $\calI(x)=\grI(m(x))$. Finally, we define
$$T(x)=\max_{I\subseteq [1,X]}\sup_{\alp \in \calI(x)}\max_{\bet\in [0,1]}\Bigl| 
\sum_{n\in I}e(\alp n^3+\bet n)\Bigr|,$$
in which the first maximum is taken over subintervals of $[1,X]$.\par

Write $\kap=\tfrac{1}{6}k!2^{k-3}$. Then \cite[Lemma 1]{HB1988} asserts that
\begin{equation}\label{3.1}
|f(\alp)|^{2^{k-3}}\ll X^{2^{k-3}-1}+X^{2^{k-3}-k+2+\eps}\sum_{h=1}^{\kap X^{k-3}}T(\alp h).
\end{equation}
Moreover, the discussion of \cite{HB1988} leading just beyond \cite[Lemma 4]{HB1988} reveals that for 
some real number $\bet=\bet(x)$, one has
\begin{equation}\label{3.2}
T(x)\ll (\log X)\sum_{l=0}^4X^{4-l}\int_{\calI(x)}\int_\bet^{\bet +X^{-1}}
\Bigl| \sum_{1\le n\le X}n^le(\xi n^3+\eta n)\Bigr| \d\eta \d\xi .
\end{equation}

\par The relation (\ref{3.2}) is the starting point for the main discussion of this section. For ease of 
discussion, and without loss of generality, we may suppose that $X$ is an integer. Our first step is to remove 
the weight $n^l$ from the innermost sum of (\ref{3.2}). On recalling (\ref{1.3}), we find by applying 
partial summation that
\begin{align*}
\sum_{1\le n\le X}n^le(\xi n^3+\eta n)&=\sum_{1\le n\le X}n^l\left( g(\xi,\eta;n)-g(\xi,\eta;n-1)\right) \\
&=X^lg(\xi,\eta;X)-\sum_{1\le n\le X-1}((n+1)^l-n^l)g(\xi,\eta;n).
\end{align*}
On substituting this relation into (\ref{3.2}), we deduce that
\begin{align*}
T(x)\ll &\, (\log X)\sum_{l=0}^4 X^{4-l}\sum_{1\le n\le X}n^{l-1}\int_{\calI(x)}
\int_\bet^{\bet +X^{-1}}|g(\xi,\eta;n)|\d\eta \d\xi \\
&\, +(\log X)\sum_{l=0}^4X^4\int_{\calI(x)}\int_\bet^{\bet+X^{-1}}|g(\xi,\eta;X)|\d\eta \d\xi ,
\end{align*}
and hence
\begin{align*}
T(x)\ll &\,X^{3+\eps}\sum_{1\le P\le X}\int_{\calI(x)}\int_\bet^{\bet+X^{-1}}|g(\xi,\eta;P)|\d\eta \d\xi 
\\
&\, +X^{4+\eps}\int_{\calI(x)}\int_\bet^{\bet+X^{-1}}|g(\xi,\eta;X)|\d\eta \d\xi .
\end{align*}
Thus we conclude that
\begin{equation}\label{3.3}
\sum_{h=1}^{\kap X^{k-3}}T(\alp h)\ll X^{4+\eps}\max_{1\le P\le X}
\sum_{h=1}^{\kap X^{k-3}}\Ttil(\alp h;P),
\end{equation}
where
$$\Ttil(x;P)=\int_{\calI(x)}\int_\bet^{\bet+X^{-1}}|g(\xi,\eta;P)|\d\eta \d\xi .$$

We must now consider the double integral $\Ttil(x;P)$, though we pause first to discuss some basic 
properties of the set $\calI(x)$. Let $x\in \dbR$, and suppose that $\calI(x)$ contains a point $\xi$ lying in 
$\grM(\frac{1}{12}P)$. Then there exists $a\in \dbZ$ and $q\in \dbN$ with $0\le a\le q\le \tfrac{1}{12}P$, 
$(a,q)=1$ and $|q\xi-a|\le \tfrac{1}{12}P^{-2}$. For all other points $\xi'\in \calI(x)$, one has
$$|q\xi'-a|\le q|\xi'-\xi|+\tfrac{1}{12}P^{-2}\le qX^{-3}+\tfrac{1}{12}P^{-2}\le \tfrac{1}{6}P^{-2}.$$
Hence we obtain the relation $\calI(x)\subseteq \grM(\tfrac{1}{6}P)$. We record for future reference also 
the bound
\begin{equation}\label{3.4}
\int_{\calI(x)}\int_\bet^{\bet+X^{-1}}\d\eta \d\xi \ll X^{-4}.
\end{equation}

\par Denote by $\grA(P)$ the set of integers $m$ with $1\le m\le X^3$ for which one has 
$\calI(mX^{-3})\cap \grM(\tfrac{1}{12}P)\ne \emptyset$, and define
$$G_m(\bet;P)=\int_{\grI(m)}\int_\bet^{\bet+X^{-1}}|g(\xi,\eta;P)|\d\eta\d\xi .$$
Thus, we have
\begin{equation}\label{3.5}
G_m(\bet;P)=\Ttil(mX^{-3};P).
\end{equation}
Then on recalling (\ref{3.4}), we find that an application of H\"older's inequality delivers the bound
$$\sum_{m\in \grA(P)}G_m(\bet;P)^8\ll X^{-28}
\sum_{m\in \grA(P)}\int_{\grI(m)}\int_\bet^{\bet+X^{-1}}|g(\xi,\eta;P)|^8\d\eta\d\xi .$$
But $\calI(mX^{-3})\subseteq \grM(\tfrac{1}{6}P)$ whenever $m\in \grA(P)$, and hence we obtain the 
relation
$$\sum_{m\in \grA(P)}G_m(\bet;P)^8\ll X^{-28}\int_{\grM(\frac{1}{6}P)}
\int_0^1|g(\xi,\eta;P)|^8\d\eta \d\xi .$$
We thus conclude from Lemma \ref{lemma2.3} and (\ref{3.5}) that
\begin{equation}\label{3.6}
\sum_{m\in \grA(P)}\Ttil(mX^{-3};P)^8=\sum_{m\in \grA(P)}G_m(\bet;P)^8\ll X^{\eps-24}.
\end{equation}

\par Meanwhile, when $\calI(mX^{-3})\cap \grM(\tfrac{1}{12}P)=\emptyset$, one has 
$\calI(mX^{-3})\subseteq \grm(\tfrac{1}{12}P)$. Then we find in a similar manner that
\begin{align*}
\sum_{\substack{1\le m\le X^3\\ m\not\in \grA(P)}}G_m(\bet;P)^{12}&\ll X^{-44}
\sum_{\substack{1\le m\le X^3\\ m\not\in \grA(P)}}\int_{\grI(m)}
\int_\bet^{\bet+X^{-1}}|g(\xi,\eta;P)|^{12}\d\eta \d\xi \\
&\ll X^{-44}\int_{\grm (\frac{1}{12}P)}\int_0^1|g(\xi,\eta;P)|^{12}\d\eta \d\xi .
\end{align*}
We thus conclude from Theorem \ref{theorem2.1} and (\ref{3.5}) that
\begin{equation}\label{3.7}
\sum_{\substack{1\le m\le X^3\\ m\not\in \grA(P)}}\Ttil(mX^{-3};P)^{12}=
\sum_{\substack{1\le m\le X^3\\ m\not\in \grA(P)}}G_m(\bet;P)^{12}\ll X^{\eps-37}.
\end{equation}

\par Next, define
$$\calT(m)=\bigcup_{l=-\infty}^\infty \grI(m+X^3l),$$
write $\calS(m)$ for the set of integers $h$ with $1\le h\le \kap X^{k-3}$ for which one has 
$\alp h\in \calT(m)$, and denote by $K(m)$ the cardinality of $\calS(m)$. We then take $\calS_1(P)$ to be 
the union of the sets $\calS(m)$ over integers $m$ with $1\le m\le X^3$ satisfying 
$\calI(mX^{-3})\cap \grM(\tfrac{1}{12}P)\ne \emptyset$, and $\calS_2(P)$ the corresponding 
union where instead $m$ satisfies $\calI(mX^{-3})\cap \grM(\tfrac{1}{12}P)=\emptyset$.\par

An application of H\"older's inequality reveals that
\begin{align*}
\Bigl( \sum_{h\in \calS_1(P)}\Ttil(\alp h;P)\Bigr)^8&\ll (X^{k-3})^7\sum_{h\in \calS_1(P)}
\Ttil(\alp h;P)^8\\
&\ll X^{7k-21}\sum_{m\in \grA(P)}K(m)\Ttil(mX^{-3};P)^8.
\end{align*}
Thus, by (\ref{3.6}), we see that
\begin{align*}
\Bigl( \sum_{h\in \calS_1(P)}\Ttil(\alp h;P)\Bigr)^8&\ll X^{7k-21}\Bigl( \max_{1\le m\le X^3}K(m)\Bigr) 
\sum_{m\in \grA(P)}\Ttil(mX^{-3};P)^8\\
&\ll X^{7k-45+\eps}\max_{1\le m\le X^3}K(m),
\end{align*}
whence
\begin{equation}\label{3.8}
\sum_{h\in \calS_1(P)}\Ttil(\alp h;P)\ll \left( X^{7k-45+\eps}\max_{1\le m\le X^3}K(m)\right)^{1/8}.
\end{equation}
Similarly, one finds that
\begin{align*}
\Bigl( \sum_{h\in \calS_2(P)}\Ttil(\alp h;P)\Bigr)^{12}&\ll (X^{k-3})^{11}
\sum_{h\in \calS_2(P)}\Ttil(\alp h;P)^{12}\\
&\ll X^{11k-33}\sum_{\substack{1\le m\le X^3\\ m\not\in \grA(P)}}K(m)\Ttil(mX^{-3};P)^{12}.
\end{align*}
Thus, by (\ref{3.7}), we obtain
\begin{align*}
\Bigl( \sum_{h\in \calS_2(P)}\Ttil(\alp h;P)\Bigr)^{12}&\ll X^{11k-33}
\Bigl( \max_{1\le m\le X^3}K(m)\Bigr) 
\sum_{\substack{1\le m\le X^3\\ m\not\in \grA(P)}}\Ttil(mX^{-3};P)^{12}\\
&\ll X^{11k-70+\eps}\max_{1\le m\le X^3}K(m),
\end{align*}
so that
\begin{equation}\label{3.9}
\sum_{h\in \calS_2(P)}\Ttil(\alp h;P)\ll \left( X^{11k-70+\eps}\max_{1\le m\le X^3}K(m)\right)^{1/12}.
\end{equation}

Suppose now that $k\ge 6$, and that $\alp\in \dbR$, $a\in \dbZ$ and $q\in \dbN$ satisfy $(a,q)=1$ and 
$|\alp-a/q|\le q^{-2}$. Then one finds from \cite[Lemma 6]{HB1988} that $K(m)\ll \Tet X^{k-3}$, where 
$\Tet=q^{-1}+X^{-3}+qX^{-k}$. Thus, on combining (\ref{3.3}), (\ref{3.8}) and (\ref{3.9}), we deduce 
that
\begin{align*}
\sum_{h=1}^{\kap X^{k-3}}T(\alp h)&\ll X^{4+\eps}
\left( (X^{8k-48}\Tet)^{1/8}+(X^{12k-73}\Tet)^{1/12}\right) \\
&\ll X^{k-2+\eps}\left( \Tet^{1/8}+(\Tet/X)^{1/12}\right) .
\end{align*}
Finally, on substituting this estimate into (\ref{3.1}), we conclude that
$$f(\alp)\ll X^{1-2^{3-k}}+X^{1+\eps}\left( \Tet^{2^{-k}}+(\Tet/X)^{\frac{2}{3}2^{-k}}\right).$$
Since $\Tet\ge X^{-3}$, the conclusion of Theorem \ref{1.2} now follows.

\bibliographystyle{amsbracket}
\providecommand{\bysame}{\leavevmode\hbox to3em{\hrulefill}\thinspace}

\end{document}